\documentclass[12pt]{article}
\usepackage{amsfonts}
\usepackage{eurosym}
\usepackage{amssymb}
\usepackage{amsmath,amssymb,latexsym}
\usepackage[dvips]{graphics}
\newtheorem{theorem}{Theorem}[section]

\newtheorem{definition}[theorem]{Definition}
\newtheorem{proposition}[theorem]{Proposition}

\newtheorem{example}[theorem]{Example}

\newtheorem{proof}[theorem]{Proof}
\numberwithin{equation}{section}

\begin{document}

\title{\textbf{Some geometric vector fields on $5$-dimensional $2$-step homogeneous nilmanifolds}}
\author{{Ghodratallah Fasihi Ramandi$\sp{a}$\thanks{fasihi@sci.ikiu.ac.ir (Gh. Fasihi Ramandi)}}\\
{\small \textit{Department of Mathematics, Faculty of Science, Imam Khomeini International University, Qazvin, Iran}}}
\date{}
\maketitle
\begin{abstract}
In this paper, we examine some geometric vector fields on $2$-step nilmanifolds of dimension $5$. We show that there is not any invariant concurrent vector field on such spaces. Our results show that in these manifolds each invariant conformal vector field is Killing and every invariant projective field is affine. Also, the space of some other invariant geometric fields such as affine, Killing, and harmonic vector fields on the manifolds are determined. 
\end{abstract}
\section{Introduction}
Two-step nilpotent Lie groups endowed with a left invariant metric, often called two-step homogeneous nilmanifolds have been studied intensively in the last twenty years. These spaces play an important role in Lie groups, geometrical
analysis and mathematical physics. A special class of two-step homogeneous nilmanifolds are Heisenberg groups. The Heisenberg groups play a crucial role in theoretical physics, and they are well understood from the viewpoint of sub-Riemannain geometry. These groups arise in the description of one-dimensional quantum mechanical systems more generally, one can consider Heisenberg groups associated to $n$-dimensional systems, and most generally, to any symplectic vector space.\\

J. Lauret classified all homogeneous nilmanifolds of dimension $3$ and $4$, up to isometry in \cite{larot}. He also, studied the structure of specific $5$-dimensional two-step nilmanifolds with $2$-dimensional center. Then, simply connected two-step nilpotent Lie groups of dimension five equipped with left invariant Riemannian metrics are classified by S. Homolya and O. Kowalski in \cite{homolya}. \\

As we already have noticed that Lie groups provide convenient example of manifolds whose geometry can be studied relatively easily, this makes Lie groups useful as spaces on which to test many geometric conjectures. In \cite{gerard} G. Walschap explored some geometric properties of Lie groups admitting a special geometric vector field. One-harmonic invariant vector fields on three-dimensional Lie groups are studied in \cite{calvin}. In this paper, we examine some geometric vector fields, such as Killing fields, conformal vector fields, projective vector fields, harmonic vector fields and concurrent vector fields on five dimensional homogeneous nilmanifolds. 
\section{Preliminaries}
Suppose $N$ is a simply connected five-dimensional two-step nilpotent Lie group endowed with a left-invariant Riemannian metric $g$ on $N$, which corresponds to an inner product $\langle ,\rangle$ on the Lie algebra $\mathfrak{n}=T_e N$ of $N$. As mentioned in
the introduction, $(N,g)$ is called a simply connected five-dimensional two-step homogeneous nilmanifold. Since, $N$ is simply connected, so the exponential mapping $\mathrm{exp}:\mathfrak{n}\longrightarrow N$ is a diffeomorphism and we need not make distinction automorphisms of $\mathfrak{n}$ and those for $N$. Note that, a Lie algebra $\mathfrak{n}$ is said to be two-step nilpotent if $[\mathfrak{n} ,\mathfrak{n}]\neq 0$  but $[\mathfrak{n},[\mathfrak{n},\mathfrak{n}]]=0$.\\

From now we consider $N$ is a simply connected two-step nilpotent Lie group of dimension five and $\mathfrak{n}$ is its Lie algebra. In order to examine geometric vector fields on these spaces we recall the classification of these spaces which is given in \cite{homolya} and their invariant Christoffel symbols which are given in \cite{salimi}.\\
{\bf Case 1: Lie algebras with one dimensional center:}
In this case there exist an orthonormal basis $\{ e_1, e_2, e_3, e_4, e_5\}$
of $\mathfrak{n}$ such that
\begin{equation}\label{bra 1}
 [e_1, e_2] = \lambda e_5 ,\quad  [e_3, e_4] = \mu e_5,
\end{equation}
where, ${e_5}$ is a basis for the center of $\mathfrak{n}$, and $\lambda \geq \mu > 0$. Also, it is considered that the other
commutators are zero. Moreover, the non-zero Christoffel symbols components are given by
\begin{align}\label{case 1}
\Gamma_{12}^{5} =-\Gamma_{21}^{5}=\dfrac{\lambda}{2} ,\qquad & \Gamma_{15}^{2}=\Gamma_{51}^{2}=-\dfrac{\lambda}{2} ,\nonumber\\
\Gamma_{25}^{1} =\Gamma_{52}^{1}=\dfrac{\lambda}{2} ,\qquad & \Gamma_{34}^{5}=-\Gamma_{43}^{5}=\dfrac{\mu}{2} ,\nonumber \\
\Gamma_{35}^{4} =\Gamma_{53}^{4}=-\dfrac{\mu}{2} ,\qquad & \Gamma_{45}^{3}=\Gamma_{54}^{3}= \dfrac{\mu}{2}.
\end{align}
{\bf Case 2: Lie algebras with two dimensional center:}
In this type, $\mathfrak{n}$ admits an orthonormal basis $\{ e_1, e_2, e_3, e_4, e_5\}$ such that
\begin{equation}\label{bra 2}
[e_1, e_2] = \lambda e_4 ,\quad  [e_1, e_3] = \mu e_5,
\end{equation}
where, $\{e_4 ,e_5\}$ is a basis for the center of $\mathfrak{n}$, the other commutators are zero and $\lambda \geq \mu > 0$. Moreover, the non-zero Christoffel symbols components are given by
\begin{align}\label{case 2}
\Gamma_{12}^{4} =-\Gamma_{21}^{4}=\dfrac{\lambda}{2} ,\qquad & \Gamma_{13}^{5}=-\Gamma_{31}^{5}=\dfrac{\mu}{2} ,\nonumber\\
\Gamma_{14}^{2} =\Gamma_{41}^{1}=-\dfrac{\lambda}{2} ,\qquad & \Gamma_{15}^{3}=\Gamma_{51}^{3}=-\dfrac{\mu}{2} ,\nonumber \\
\Gamma_{24}^{1} =\Gamma_{42}^{1}=\dfrac{\lambda}{2} ,\qquad & \Gamma_{35}^{1}=\Gamma_{53}^{1}= \dfrac{\mu}{2}.
\end{align}
{\bf Case 3: Lie algebras with three dimensional center:}
The Lie algebra structure of this case is as follows.\\
The Lie algebra, $\mathfrak{n}$ admits an orthonormal basis $\{ e_1, e_2, e_3, e_4, e_5\}$ such that for $\lambda >0$
\begin{equation}\label{bra 3}
[e_1, e_2] = \lambda e_3,
\end{equation}
where, $\{e_3,e_4 ,e_5\}$ is a basis for the center of $\mathfrak{n}$, the other commutators are zero. Moreover, the non-zero Christoffel symbols components are given by
\begin{align}\label{case 3}
&\Gamma_{12}^{3} =-\Gamma_{21}^{3}=\dfrac{\lambda}{2} ,\qquad  \Gamma_{13}^{2}=\Gamma_{31}^{2}=-\dfrac{\lambda}{2} ,\nonumber \\
&\Gamma_{23}^{1} =\Gamma_{32}^{1}=\dfrac{\lambda}{2}.
\end{align}
\section{Geometric Vector Fields}
In this section, we will look at some geometric vector fields on a manifold with a linear connection. These concepts are needed at the next section. For details on most of the ideas will touch on here, you can consult \cite{poor}.
\subsection*{Killing Vector Field}
\begin{definition}
A vector field $X$ on a Riemannian manifold $(M,g)$ is said to be a Killing field if and only if $\mathcal{L}_X g=0$, where, $\mathcal{L}_X$ stands for the Lie derivative with respect to $X$. 
\end{definition}
In particular, a Killing field is divergence free and we have the following equation which is known as Killing's equation.
\[\langle \nabla_U X ,V \rangle +\langle U,\nabla_V X \rangle=0,\qquad \forall U,V\in \mathcal{X}(M).\] 
Above definition shows that a Killing vector field on a Riemannian manifold $(M,g)$  preserves the metric. In fact, Killing fields are the infinitesimal generators of isometries; that is, flows generated by Killing fields are isometries of the manifold. More simply, the flow of a Killing filed generates a symmetry, in the sense that moving each point on an object the same distance in the direction of the Killing vector will not distort distances on the object. Also, a typical use of Killing fields is to express a symmetry in space-time manifolds.
\begin{example}
Let $ M=\{(x,y)\in \mathbb {R}^2|y>0\}$ is the upper half plane equipped metric $ g=\dfrac{(dx^{2}+dy^{2})}{y^2}$. The pair $( M,g)$  is typically called the hyperbolic plane and has Killing vector field $X=\dfrac{\partial}{\partial x}$ (using standard coordinates). This should be intuitively clear since the covariant derivative $ \nabla _X g$ transports the metric along an integral curve generated by the vector field. 
\end{example}
\subsection*{Harmonic Vector Field}
One of the most important operators determined by a Riemannian metric is Laplace-Beltrami operator. The kernel of this second-order differential operator is related to De Rham cohomology of the underlying manifold. We define this operator and introduce the notion of harmonic vector fields in what follows.
\\
The $\mathrm{div}$ operator on the space of covariant tensor fields on a Riemannian manifold $(M,g)$ maps the subspace $A^p (M)$ of differential $p$-form to the subspace $A^{p-1}(M)$. The restriction of $\mathrm{div}$ operator to $A(M)$ the algebra of differential forms on $M$ is important and is denoted by a special symbol.
\begin{definition}
Define $\delta (\omega) =-\mathrm{div}(\omega)$, for $\omega \in A(M)$.
\end{definition}
\begin{definition}
Suppose that $d$ is the exterior derivative operator, we define the Laplace-Beltrami (or Laplacian) operator $\triangle :A^p(M)\longrightarrow A^p(M)$ by
\[\triangle (\omega)=(d\circ \delta +\delta \circ d)(\omega).\]
We say a $p$-form is harmonic if it is in the kernel of $\triangle$.
\end{definition}
\begin{definition}
A vector field $X$ in a Riemannian manifold $(M,g)$ is called harmonic if the metric dual $1$-form $X^\flat =g(X,.)$ is harmonic.
\end{definition}

\subsection*{Conformal Vector Field}
Killing vector fields preserve the metric, so it is natural to ask is there some vector fields on a Riemannian manifold which preserve the metric up to a constant factor. This question leads us to the notion of conformal vector fields.
\begin{definition}
Riemannian metrics $g_1$ and $g_2$ on a manifold $M$ is said to be conformally equivalent if there exists $f\in C^\infty (M)$ such that $g_2=fg_1$.
\end{definition}
Let $g_1$ and $g_2$ be conformally equivalent metrics on $M$, then one can easily check that  for each $p\in M$, $g_1$ and $g_2$ induces the same angle measure on $T_p M$. Conversely, an angle measure determines a conformal equivalence class of inner product on $T_pM$.  
\begin{definition}
A diffeomorphism $f:(M,g)\longrightarrow (M,g)$ on a Riemannian manifold $M$ is said to be conformal transformation on $M$ if $f^*(g)$ is conformally equivalent to $g$. 
\end{definition}
\begin{definition}
A vector field $X$ on a Riemannian manifold $(M,g)$ is said to be conformal vector field if whose local $1$-parameter groups consists of local conformal maps. 
\end{definition}
According to above definition, conformal vector fields are the infinitesimal generators of conformal transformations. An isometric map of Riemannian manifolds is a conformal map. Also, Killing fields of a Riemannian manifold are conformal vector fields. In fact, conformal vector fields are generalization of Killing fields.\\
It is proven that the group of all conformal transformations of a connected Riemannian manifold is a Lie group and its Lie algebra is isomorphic to Lie algebra of complete conformal vector fields on $M$.\\
 We will need the following proposition in the next section.
\begin{proposition}\cite{poor}
The following statements are equivalent for a vector field $X$ on a Riemannian manifold $(M^n,g)$:\\
(i) $X$ is a conformal field,\\
(ii) $\mathcal{L}_X g=2hg$ for some $h\in C^{\infty}(M)$,\\
(iii) $\mathcal{L}_X g=\dfrac{2\mathrm{div}(X)}{n}g$.
\end{proposition}
\subsection*{Affine and Projective Vector Field}
\begin{definition}
A map $f:(N,\nabla^N)\longrightarrow (M,\nabla^M)$ of manifolds with linear connections is called affine if
\[f_* \nabla^N_X Y =\nabla^M_X f_* Y,\qquad X,Y\in \mathcal{X}(M).\]
An Affine transformation of $(M,\nabla)$ is an affine diffeomorphism of $M$.
\end{definition}
\begin{definition}
An affine vector field on a Riemannian manifold $(M,g)$ is an element $X\in\mathcal{X}(M)$ such that the local $1$-parameter group of $X$ consists of local affine maps of $(M,\nabla)$, where, $\nabla$ is the Levi-Civita connection of $M$.
\end{definition}
Affine vector fields preserve the geodesic structure of semi-Riemannian manifolds whilst also preserving the affine parameter. There exists another smooth vector field on a semi-Riemannain manifolds whose flow preserves the geodesic structure without necessarily preserving the affine parameter of any geodesic. In fact, the flow of a projective field maps geodesics smoothly into geodesics without preserving the affine parameter. 
\begin{definition}
A map $f:(N,\nabla^N)\longrightarrow (M,\nabla^M)$ of manifolds with torsion-free connections is called projective if for each geodesic $\gamma$ of $\nabla^N$, $f\circ \gamma$ is a reparametrization of a geodesic of $\nabla^M$. A projective transformation of $(M,\nabla)$is a diffeomorphism $f:(M,\nabla)\longrightarrow (M,\nabla)$ which is projective.
\end{definition}
Similar to previous definitions, we can define projective vector field as follows.
\begin{definition}
A vector field $X$ on a Riemannian manifold $(M,g)$ with associated Levi-Civita connection $\nabla$ is projective if its local $1$-parameter group consists of local projective transformations.
\end{definition}
We shall use the following proposition in the next section to determine all projective vector fields on simply connected two-step homogeneous nilmanifolds of dimension $5$.
\begin{proposition}\cite{poor}
A vector field $X$ on a Riemannian manifold $(M,g)$ is projective if and only if there exists $\alpha \in A^1(M)$ which will be called the associated 1-form, such that
\[ (\mathcal{L}_X \nabla)(U,V)=\alpha (U)V+\alpha (V)U,\]
where, $\nabla$ is the Levi-Civita connection of the meter $g$. Furthermore, $X$ is affine if and only if $\alpha=0$.
\end{proposition}
\subsection*{Concurrent Vector Field}
It was proved in \cite{22} that if the holonomy group of a Riemannian $n$-manifold $(M,g)$ leaves a point invariant, then there exists a vector field $X$ on $M$ which satisfies
\[\nabla_Y X=Y\]
for any vector field $Y$ on $M$, where $\nabla$ denotes the Levi-Civita connection of $M$.
\begin{definition}
A vector field $X$ on a Riemannian Manifold $(M,g)$ is called concurrent vector field if for each vector field $Y$ on $M$ satisfies the following equation.
\[\nabla_Y X=Y,\] 
where $\nabla$ denotes the Levi-Civita connection of $M$.
\end{definition}
Geometry of such vector fields have been studied by many mathematicians (see \cite{7}, \cite{15} and \cite{20}). In \cite{soliton} a complete classification of Ricci solitons with concurrent potential field is done. Concurrent vector fields have also been studied in Finsler geometry since the beginning of 1950s (see \cite{16} and \cite{21}). In the next section we show there is no concurrent vector field on simply connected two-step homogeneous nilmanifolds of dimension $5$.
\section{Main Results}
In this section we present our main results on two-step homogeneous nilmanifolds of dimension 5.
\begin{theorem}
There is not any left-invariant concurrent vector field on two-step homogeneous nilmanifolds of dimension 5.
\end{theorem}
\begin{proof}
Suppose that $X\in \mathfrak{n}$ is a left-invariant vector field on $N$ which is concurrent with respect to Levi-Civita connection. We prove the theorem in separately the cases where the dimension of center is $1,2$ or $3$.\\ 
{\bf Case 1:} Let $X=\sum_{i=1}^{n}a_i e_i$ is a concurrent vector field, then
\[\nabla_{e_5} X=e_5,\qquad \nabla_{e_1} X=e_1.\]
According to (\ref{case 1}), we can write
\begin{align*}
e_5 &=\dfrac{\lambda}{2}(a_2 e_1 -a_1e_2) +\dfrac{\mu}{2}(a_4 e_3-a_3 e_4),\\
e_1&=\dfrac{\lambda}{2}(a_2 e_5 -a_5e_2).
\end{align*}
Because $\{e_i \}_{i=1}^{5}$ is independent, therefore $a_i =0$ for $1\leq i\leq 5$ which gives us the contradiction $X=0$.\\
{\bf Case 2:}
In this type consider the vector field $X=\sum_{i=1}^{n}a_i e_i$ to be concurrent with respect to the Levi-Cicita connection, then
\[\nabla_{e_1} X=e_1,\qquad \nabla_{e_2} X=e_2,\]
but, according to (\ref{case 2}), we have
\begin{align*}
e_1&=\dfrac{\lambda}{2}(a_2 e_4 -a_4e_2) +\dfrac{\mu}{2}(a_3 e_5-a_5 e_3),\\
e_2 &=\dfrac{\lambda}{2}(a_4 e_1 -a_1e_4).
\end{align*}
Since $\{e_i \}_{i=1}^{5}$ is independent, therefore $a_i =0$ for $1\leq i\leq 5$ which gives us the contradiction $X=0$.\\
{\bf Case 3:} Let $X=\sum_{i=1}^{n}a_i e_i$ be an arbitrary vector field on $N$. By (\ref{case 3}), we have $$\nabla_{e_4} X=0.$$ Hence, in this type, there is not any concurrent vector field on $N$.
\end{proof}

Notice that if a left-invariant vector field $X=\sum_{k=1}^{5}a_k e_k$ on $(N,\nabla)$ is projective vector field then there exists a functional $f\in \mathfrak{n}^*$, such that 
\[(\mathcal{L}_X \nabla )(e_i ,e_j)=f(e_i)e_j+f(e_j)e_i.\]
Since, $\mathfrak{n}$ is two-step nilpotent Lie algebra, easy computation shows that 
\[(\mathcal{L}_X \nabla )(e_i ,e_j) =\sum_{k,l=1}^{5}a_k \Gamma_{ij}^l [e_k,e_l].\]
\begin{theorem}
Denote by $\mathfrak{h}$ the center of $\mathfrak{n}$. Each left-invariant projective vector field $X\in \mathfrak{n}$ is affine and $X$ is affine if and only if $X\in \mathfrak{h}$.
\end{theorem}
\begin{proof}
Suppose that $X\in \mathfrak{n}$ is a left-invariant projective vector field on $N$. We prove the theorem in separately the cases where the dimension of center is $1,2$ or $3$.\\ 
{\bf Case 1:} Let $X=\sum_{i=1}^{5}a_i e_i$ is a projective vector field, then according to (\ref{bra 1}) we can write,
\[(\mathcal{L}_X \nabla )(e_i ,e_j)=(a_1 \lambda \Gamma_{ij}^{2}-a_2 \lambda \Gamma_{ij}^{1}+a_3 \mu \Gamma_{ij}^{2}-a_4 \mu \Gamma_{ij}^{3})e_5.\]
The formulas (\ref{case 1}) show that
\begin{align*}
(\mathcal{L}_X \nabla )(e_1 ,e_5) =-a_1 \dfrac{\lambda^2}{2}e_5,\quad &(\mathcal{L}_X \nabla )(e_2 ,e_5)=-a_2 \dfrac{\lambda^2}{2}e_5 ,\\
(\mathcal{L}_X \nabla )(e_3 ,e_5) =-a_3 \dfrac{\mu^2}{2}e_5,\quad &(\mathcal{L}_X \nabla )(e_4 ,e_5)=-a_4 \dfrac{\mu^2}{2}e_5 .
\end{align*}
So, $f\in \mathfrak{n}^*$ (where, $\mathfrak{n}^*$ shows the dual space on $\mathfrak{n}$) must satisfy
\begin{align*}
f(e_1) =-a_1 \dfrac{\lambda^2}{2},\quad &f(e_2)=-a_2 \dfrac{\lambda^2}{2} ,\\
f(e_3)=-a_3 \dfrac{\mu^2}{2},\quad &f(e_4)=-a_4 \dfrac{\mu^2}{2} ,\quad f(e_5)=0.
\end{align*}
But, we also have
\begin{align*}
0=& (\mathcal{L}_X \nabla )(e_1 ,e_2) =f(e_1)e_2 +f(e_2)e_1, \\
0=& (\mathcal{L}_X \nabla )(e_3 ,e_4) = f(e_3)e_4 +f(e_4)e_3 .
\end{align*}
Because $e_i$s are independent, therefore we have $a_1 =a_2=a_3=a_4=0$ and $f$ is identically zero.\\
 
{\bf Case 2:}
In this case, suppose that $X=\sum_{i=1}^{5}a_i e_i$ is a projective vector field, then according to (\ref{bra 2}) we can write,
\[(\mathcal{L}_X \nabla )(e_i ,e_j)=(a_1 \lambda \Gamma_{ij}^{2}-a_2 \lambda \Gamma_{ij}^{1})e_4+(a_1 \mu \Gamma_{ij}^{3}-a_3 \mu \Gamma_{ij}^{1})e_5.\]
The formulas (\ref{case 2}) show that
\begin{align*}
&(\mathcal{L}_X \nabla )(e_1 ,e_4) =-a_1 \dfrac{\lambda^2}{2}e_4,\quad (\mathcal{L}_X \nabla )(e_1 ,e_5)=-a_1 \dfrac{\mu^2}{2}e_5 ,\\
&(\mathcal{L}_X \nabla )(e_2 ,e_4) =-a_2 \dfrac{\lambda^2}{2}e_4-a_3 \dfrac{\mu \lambda}{2}e_5,\\
&(\mathcal{L}_X \nabla )(e_3 ,e_5)=-a_2 \dfrac{\mu \lambda}{2}e_4 -a_3\dfrac{\mu^2}{2}e_5.
\end{align*}
So, $f\in \mathfrak{n}^*$ must satisfy
\begin{align*}
&f(e_1) =-a_1 \dfrac{\lambda^2}{2},\quad  f(e_4)=0,\\
&f(e_2)=-a_2 \dfrac{\lambda^2}{2},\quad a_3=0,\\
&f(e_3) =-a_3 \dfrac{\mu^2}{2},\quad a_2=0,\\
&f(e_1)=-a_1 \dfrac{\mu^2}{2},\quad  f(e_5)=0.
\end{align*}
But, we also have
\begin{align*}
0= (\mathcal{L}_X \nabla )(e_1 ,e_3) =f(e_1)e_3 +f(e_3)e_1.
\end{align*}
Because $e_i$s are independent, therefore we have $a_1 =a_2=a_3=0$ and $f$ is identically zero.\\
 
{\bf Case 3:} By considering $X=\sum_{i=1}^{5}a_i e_i$ as a projective vector field, then according to (\ref{bra 3}) we have,
\[(\mathcal{L}_X \nabla )(e_i ,e_j)=(a_1 \lambda \Gamma_{ij}^{2}-a_2 \lambda \Gamma_{ij}^{1})e_3.\]
The formulas (\ref{case 3}) show that
\begin{align*}
(\mathcal{L}_X \nabla )(e_1 ,e_3) =-a_1 \dfrac{\lambda^2}{2}e_3,\quad &(\mathcal{L}_X \nabla )(e_1 ,e_4)=0 ,\\
(\mathcal{L}_X \nabla )(e_2 ,e_3) =-a_2 \dfrac{\mu^2}{2}e_3,\quad &(\mathcal{L}_X \nabla )(e_3 ,e_5)=0 .
\end{align*}
So, $f\in \mathfrak{n}^*$ must satisfy
\begin{align*}
&f(e_1) =-a_1 \dfrac{\lambda^2}{2},\quad f(e_3)=0 ,\\
&f(e_1) =0 ,\quad \qquad f(e_4)=0 ,\\
&f(e_3) =0,\quad\qquad f(e_5)=0 ,\\
&f(e_2)=-a_2 \dfrac{\mu^2}{2},\quad f(e_3)=0.
\end{align*}
But, we also have
\begin{align*}
0= (\mathcal{L}_X \nabla )(e_2 ,e_4) =f(e_2)e_4 +f(e_4)e_2,
\end{align*}
Because $e_i$s are independent, therefore we have $a_1 =a_2=0$ and $f$ is identically zero.\\
By above computations the converse is clear.
\end{proof}

As mentioned before, a vector field $X$ on a Riemannain manifold $(M,g)$ is conformal if and only if $\mathcal{L}_X g=\dfrac{2\mathrm{div}(X)}{n}g$. Let $X=\sum_{i=1}^{n}a_i e_i$ be an arbitrary left-invariant vector field on $N$, then $\mathrm{div}(X)$ is computed as follows,
\[\mathrm{div}(X) =\sum_{i=1}^{5}\langle \nabla_{e_i} X ,e_i \rangle .\]
In all cases where the dimension of center is $1,2$ or $3$, the formulas (\ref{case 1}), (\ref{case 2} ) and (\ref{case 3} ) show that the right hand side of the above formula is identically zero. Hence, we have the following theorem.
\begin{theorem}
Every left-invariant conformal vector field on $N$ is Killing field. 
\end{theorem}

The above theorem indicates that the subspace of Killing fields is equal to the space of conformal vector fields in two-step homogeneous nilmanifolds of dimension 5. In the next theorem we investigate on Killing fields in theses spaces. Straightforward computations show that a left-invariant vector field $X\in \mathfrak{n}$ is Killing if and only if
\[ \langle \nabla_U X ,V\rangle+\langle U, \nabla_V X\rangle=0
\] 
By means of an orthonormal basis $\{e_i\}_{i=1}^{5}$ for $\mathfrak{n}$, let $X=\sum_{k=1}^5 a_k e_k$, local computations of the above formula is as follows.
\begin{equation}\label{killing}
\langle \nabla_{e_i} X ,e_j\rangle+\langle e_i, \nabla_{e_j} X\rangle =\sum_{k=1}^{n} a_k (\Gamma_{ik}^j +\Gamma_{jk}^{i}).
\end{equation}
\begin{theorem}
The Lie algebra of all left-invariant Killing vector fields of $N$ is four-dimensional.
\end{theorem}
\begin{proof}
Suppose that $\{e_i\}_{i=1}^{5}$ is an orthonormal basis of $\mathfrak{n}$ and $X=\sum_{k=1}^5 a_k e_k$ is a left-invariant Killing field on $N$. We prove the theorem in separately the cases where the dimension of center is $1,2$ or $3$.\\ 
{\bf Case 1:} In this type, by (\ref{case 1} ) we have 
\begin{align*}
0&=\sum_{k=1}^{n} a_k (\Gamma_{ik}^j +\Gamma_{jk}^{i})\\
&= \lambda (a_2 -a_1) +\mu (a_4 -a_3).
\end{align*}
According to above equality, we have three degrees of freedom in choosing $\{a_i\}_{i=1}^{4}$ and $a_5$ could be choose arbitrarily.\\
{\bf Case 2:} In this case, the formulas (\ref{case 2} ) show that
\begin{align*}
0&=\sum_{k=1}^{n} a_k (\Gamma_{ik}^j +\Gamma_{jk}^{i})\\
&= \lambda (a_2 -a_1) +\mu (a_3 -a_1).
\end{align*}
Hence, we have two degrees of freedom in choosing $\{a_i\}_{i=1}^{3}$. Also, $a_4$ and $a_5$ could be choose arbitrarily.\\
{\bf Case 3:} In the case, according to (\ref{case 3} ) we have the following equality
\begin{align*}
0&=\sum_{k=1}^{n} a_k (\Gamma_{ik}^j +\Gamma_{jk}^{i})\\
&= \lambda (a_2 -a_1).
\end{align*}
So, $a_1=a_2$ and other coefficients could be choose arbitrarily.
\end{proof}

Now, we want to examine harmonic vector fields on $N$.
\begin{theorem}
Denote by $\mathfrak{h}$ the center of $\mathfrak{n}$. A left-invariant vector field $X\in \mathcal{X}(N)$ is harmonic if and only if \\
a) $X\in \mathfrak{h}^\perp$ ,in the cases where dimension of $\mathfrak{h}$ is 1 or 2.\\
b) $\langle X, e_3 \rangle=0$, in the case where dimension of $\mathfrak{h}$ is 3.
\end{theorem}
\begin{proof}
Let $X=\sum_{i=1}^{5}a_i e_i$ is an arbitrary left-invariant vector field on $(N,g)$ which is harmonic with respect to the Levi-Civita connection of $N$. As mentioned before $\mathrm{div}(X)=0$, so we have
\[\delta (X^\flat)=-\mathrm{div}(X^\flat)=-\mathrm{div}(X)=0.\]
Hence, a left-invariant vector field $X=\sum_{i=1}^{n}a_i e_i$ is harmonic if and only if $(\delta \circ d)(X^\flat)=0$. An easy computation shows that
\[\delta(dX^\flat)(e_j)=-\sum_{i,k}^{5}\big( \Gamma_{ii}^k \langle X, [e_k ,e_j]\rangle +\Gamma_{ij}^k \langle X, [e_i ,e_k]\rangle\big).\]
{\bf Case 1:} In this case, the only non-zero component of $\delta(dX^\flat)$ is given by
\[\delta(dX^\flat)(e_5) =a_5 (\dfrac{\lambda^2}{2}+\dfrac{\mu^2}{2}). \]
So, $X$ is harmonic if and only if $X\in \mathfrak{h}^\perp$.\\
{\bf Case 2:} In this case, we get
\[\delta(dX^\flat)(e_4) =-a_4 \lambda^2 , \quad \delta(dX^\flat)(e_5) =a_5 \mu^2,\]
and the other coefficients are zero. We must have $a_4 =a_5 =0$.
\\
{\bf Case 3:} In this case, it can be seen that
\[\delta(dX^\flat)(e_3) =a_3 \lambda^2,\]
and the other coefficients are zero. So, $X$ is harmonic if and only if $a_3 =0$.
\end{proof}

\begin{thebibliography}{99}
\bibitem{calvin} E. Calvino-Louzao, J. Seoane-Bascoy, M.E. Vazquez-Abal and R. Vazquez-Lorenzo, One-harmonic invariant vector fields on three-dimensional Lie groups, J. Geom. Phys. 62(2012), 1532-1547.
\bibitem{7}{B.-Y. Chen and S. Deshmukh, Geometry of compact shrinking Ricci solitons, Balkan J. Geom. Appl. 19 (2014), 13-21.}
\bibitem{soliton}{B.-Y. Chen and S. Deshmukh, Ricci solitons and concurrent vector fields, Balkan Journal of Geometry and Its Applications 20 (2015),14-25.}
\bibitem{homolya}{S. Homolya and O. Kowalski, Simply connected two-step homogeneous nilmanifolds of dimension 5, Note Mat. 26 (2006), 69-77.}
\bibitem{larot}{J. Lauret, Homogeneous nilmanifolds of dimension $3$ and $4$, Geometriae Dedicatea, 68 (1997), 145-155.}
\bibitem{lee} {J. M. Lee, Introduction to Smooth Manifolds, Springer, 2000.} 
\bibitem{15}{M. F. Lopez and E. Garcia-Rio, A remark on compact Ricci solitons, Math. Ann. 340 (2008), 893-896.}
\bibitem{16}{M. Matsumoto and K. Eguchi, Finsler spaces admitting a concurrent vector field, Tensor (N.S.) 28 (1974), 239-249.}
\bibitem{20}{M. Petrovi\'{c}, R. Rosca and L. Verstraelen, Exterior concurrent vector fields on Riemannian manifolds. I. Some general results, Soochow J. Math. 15 (1989), 179-187.}
\bibitem {poor}{W. A. Poor, Differential Geometric Structures, McGraw-Hill, 1981.}
\bibitem{salimi}{H. R. Salimi Moghaddam, On the Randers metrics on two-step homogeneous nilmanifolds
of dimension five, Int. J. Geom. Methods Mod. Phys. 8 (2011), 501-510.}
\bibitem{21}{S.-I. Tachibana, On Finsler spaces which admit a concurrent vector field, Tensor (N.S.) 1, (1950), 1-5.}
\bibitem{22}{K. Yano, Sur le parall\'{e}lisme et la concurance dans l'espace de Riemann, Proc. Imp. Acad. Tokyo, 19 (1943), 189-197.}
\bibitem{gerard}{G. Walschap, Geoemtric vector fields on Lie groups, Differential Geometry and its Applications, 7 (1997), 219-230.}
\end{thebibliography}

\end{document}